\documentclass{amsart}
\usepackage{graphicx} 
\usepackage{subfig}
\usepackage{tabularx}

\usepackage{transparent} 
\usepackage{eso-pic} 
\usepackage{subfig} 
\usepackage{tikz} 
\usetikzlibrary{}
\graphicspath{{./Images/}} 
\usepackage{caption} 
\usepackage{xcolor} 
\usepackage{amsthm,thmtools,xcolor} 
\usepackage{float}

\usepackage{geometry}
\usepackage[utf8]{inputenc}
\usepackage[T1]{fontenc}
\usepackage{mathtools}
\usepackage[thinc]{esdiff}

\usepackage{amsmath}
\usepackage{amsthm}
\usepackage{amssymb}
\usepackage{amsfonts}
\usepackage{bm}
\usepackage[overload]{empheq} 
\usepackage{fix-cm} 
\usepackage{algorithm}
\usepackage{algorithmic}

\newtheorem{lemma}{Lemma}[section]  
\newtheorem{theorem}[lemma]{Theorem}  

\usepackage[colorlinks=true,linkcolor=black,anchorcolor=black,citecolor=black,filecolor=black,menucolor=black,runcolor=black,urlcolor=black]{hyperref} 
\usepackage{cleveref}
\crefname{lemma}{Lemma}{Lemmas}
\crefname{theorem}{Theorem}{Theorems}
\crefname{proposition}{Proposition}{Propositions}
\crefname{remark}{Remark}{Remarks}
\crefname{problem}{Problem}{Problems}
\crefname{corollary}{Corollary}{Corollaries}
\crefname{definition}{Definition}{Definitions}

\crefname{section}{Section}{Sections}
\usepackage{graphicx}
\usepackage{float}
\usepackage[square, numbers, sort&compress]{natbib} 

\def\R{{\mathbb{R}}}
\def\erre{{\mathbb{R}}}
\def\enne{{\mathbb{N}}}
\def\esse{{\mathbb{S}}}
\def\eps{{\varepsilon}}

\title{Quantitative stability of the Rossby--Haurwitz waves of degree two for the Euler equation on $\mathbb{S}^2$}
\author{Matias G. Delgadino}
\address{Department of Mathematics, University of Texas at Austin}
\email{matias.delgadino@math.utexas.edu}
\author{Luca Melzi}
\address{Department of Mathematics, Imperial College London}
\email{l.melzi24@imperial.ac.uk}
\date{\today}

\begin{document}

\begin{abstract}
    We show that the degree-2 Rossby--Haurwitz travelling waves on the Euler equation on $\mathbb{S}^2$ are orbitally stable. Our proof is short, quantitative, and conceptually easy to follow.
\end{abstract}

\maketitle

\section{Introduction}
Our starting point is the incompressible Euler equation in vorticity form on the sphere $\mathbb{S}^2\subset\R^3$:
\begin{align}
\begin{cases}
    \partial_t\omega+\nabla^\perp\psi\cdot \nabla\omega=0&\text{on }[0,\infty)\times\mathbb{S}^2,\\
    \Delta\psi=\omega&\text{on }[0,\infty)\times\mathbb{S}^2,\\
    \omega|_{t=0}=\omega_0&\text{on }\mathbb{S}^2,
    \label{eq:euler_rotating}
\end{cases}
\end{align}
where $\Delta$ is the Laplace--Beltrami operator on $\mathbb{S}^2$, $\nabla^\perp$ is the curl operator, $\psi=\Delta^{-1}\omega$ is the stream function. We refer to \cite{Taylor2016} for long time existence and uniqueness of solutions to \eqref{eq:euler_rotating} under regularity hypothesis of the initial vorticity $\omega_0$.

A natural class of steady states to the Euler equation are given by the eigenfunctions of the Laplacian:
\begin{equation}\label{eq:eigenfunction}
    -\Delta Y_{j}= j(j+1) Y_{j}.
\end{equation}
The object we study are the Rossby--Haurwitz travelling waves, which are obtained by adding a pure rotation to an eigenfunction \eqref{eq:eigenfunction}, see \cite{rossby1939,haurwitz1940} for the classical references and \cite{Costantin2022} for modern take. To be more precise, we denote $R_p^u\in \mathbb{SO}(3)$ the rotation around the vector $p\in\esse^2$ with angle $u\in \R$. For any $\Omega\in \R$, we have that
\begin{equation}\label{eq:RH}
\omega^{RH}_t(x)=\Omega\, p\cdot x+Y_{j}\left(R_p^{-ct}x\right)    
\end{equation}
is a periodic travelling solution, with travelling speed
\begin{equation}\label{eq:travellingspeed}
    c=\Omega\left(\frac{1}{2}-\frac{1}{j(j+1)}\right),
\end{equation}
see \cite[Appendix A]{cao2023stabilitydegree2rossbyhaurwitzwaves}. When the eigenfunction is invariant under the rotation $Y_j(R_p^{-ct}x)=Y_j(x)$, the travelling wave \eqref{eq:RH} is also a steady state of \eqref{eq:euler_rotating}. We will refer to this case as zonal solutions.

From a modelling perspective, the rotation term $\Omega p\cdot x$ in \eqref{eq:RH} represents the Coriolis force arising from the Earth's rotation about the axis $p$ with angular velocity $\Omega$. Rossby--Haurwitz (RH) waves are fundamental in meteorology for describing large-scale atmospheric dynamics \cite{holton2012}. These waves are observed not only in Earth's atmosphere but also in the stratospheres of other planets, including Jupiter, Saturn, Uranus, and Neptune. The mathematical analysis of RH waves provides insight into geophysical phenomena such as the ozone hole and planetary features like Jupiter’s Great Red Spot. Moreover, understanding the stability properties of RH waves is crucial for assessing the reliability of long-term weather prediction models.

The stability of Rossby--Haurwitz travelling wave solutions \eqref{eq:RH} depends heavily on the degree of the eigenfunction \eqref{eq:eigenfunction}. It was shown in \cite{constantin2025onsetinstabilityzonalstratospheric} that the stability of zonal waves of degree 3 depends on the rotation speed $\Omega$. For $\Omega\in(\omega_{cr}^-,\omega_{cr}^+)$ normalized zonal waves of degree 3 are unstable, while they are stable for $\Omega\in \R\setminus (\omega_{cr}^-,\omega_{cr}^+)$. In general, understanding the stability of higher degree zonal or non-zonal waves is an open problem.

In this article we focus on the case of degree 2 waves. Qualitative stability of the degree-2 zonal waves with $\Omega\ne 0$ was shown by Constantin--Germain in \cite{Costantin2022}. One of the main contributions of this paper is to show the quantitative stability of these solutions. More specifically, we show the following statement. 
\begin{theorem}\label{thm:1}
Let $\bar\omega^{RH}$ be a degree-2 zonal wave with $\Omega\ne 0$, then there exists $C, \bar\eps>0$ such that if $\|\omega_0-\bar{\omega}^{RH}\|<\bar\eps$, then 
\begin{equation}\label{eq:1ststability}
    \|\omega_t-\bar{\omega}^{RH}\|_2\leq C \|\omega_0-\bar{\omega}^{RH}\|^{1/2}_2,
\end{equation}
where $\{\omega_t\}_{t\ge 0}$ is the solution to Euler equation in $\esse^2$ with initial condition $\omega_0$ \eqref{eq:euler_rotating}.
\end{theorem}

The case $\Omega=0$ is degenerate, as the solution has no preferred axis of rotation. Hence adding any small rotation around a chosen axis produces the solution to slowly start rotating around this axis. To be specific for any $p\in\esse^2$, we can take
$$
\omega_0=Y_2(x)\qquad\mbox{and}\qquad\omega_0^\eps=\eps p\cdot x+Y_2(x),
$$
which satisfy $\|\omega_0^\eps-\omega_0\|_2\lesssim\eps$ and
\begin{equation}\label{eq:degeneracy_case}
\|\omega_t^\eps-\omega_t\|_2^2= \eps^2 + \|Y_2(R_p^{-\eps\frac{t}{3}
}x)-Y_2(x)\|_2^2,    
\end{equation}
which is of order 1 if $t\sim 1/\eps$ and we chose $p\in\esse^2$ such that $Y_2$ is not invariant under rotations. For $\Omega= 0$, qualitative stability of degree-2 waves, up to rotation through any axis, was shown by Cao--Wang--Zuo \cite{cao2023stabilitydegree2rossbyhaurwitzwaves}. In this case, we can also show a quantitative result. 
\begin{theorem}\label{thm:2}
Let $\bar{\omega}=Y_2$ be an eigenfunction in the second shell, which is a steady state of Euler equation, then there exists $C, \bar\eps>0$ such that if $\|\omega_0-\omega\|_2<\bar\eps$, then
\begin{equation}\label{eq:3rdstability}
    \inf_{p\in\esse^2,\,s\in\R }\|\omega_t(x)-\bar{\omega}(R^s_p x)\|_2\leq C  \|\omega_0-\bar{\omega}\|_2,
\end{equation}
where $\{\omega_t\}_{t\ge 0}$ is any solution to the Euler equation in $\esse^2$ \eqref{eq:euler_rotating}.
\end{theorem}

We note that a salient difference between \cref{thm:1} and \cref{thm:2} is the exponent in the stability, and is related to a degeneracy in our method of proof. In general, the exponent of stability of non-zonal degree-2 waves depends on which Fourier modes are non-trivial.
\begin{theorem}\label{thm:main}
    Let $\Omega\ne 0$ and
    \begin{equation*}
        \omega^{RH}_t(x)=\Omega x_3 +\sum_{m=-2}^2c_{2,m} Y_{2,m}\left(R_{e_3}^{-\Omega\frac{t}{3}}x\right)
    \end{equation*} be degree-2 Rossby--Haurwitz wave \eqref{eq:RH} rotating around the north pole, where $Y_{j,m}$ denote the spherical harmonics of degree $j$ and order $m$, see \cref{subsec:spherical_harmonics}. There exists $C, \bar\eps>0$ such that if $\|\omega_0-\omega_0^{RH}\|_2<\bar\eps$, then
    \begin{equation}\label{eq:2ndstability}
        \inf_{s\in[0,\infty)}\|\omega_t-\omega_s^{RH}\|_2\leq C  \|\omega_0-\omega_0^{RH}\|_2^{1/2},
    \end{equation}
    where $\{\omega_t\}_{t\ge 0}$ is any solution to the Euler equation in $\esse^2$ with initial condition $\omega_0$ \eqref{eq:euler_rotating}.

    Moreover, if
    \begin{equation*}
        |c_{2,-1}|^2+|c_{2,1}|^2>0,
    \end{equation*}
    then the stability estimate \eqref{eq:2ndstability} can be improved to
    \begin{equation}
        \inf_{s\in[0,\infty)}\|\omega_t-\omega_s^{RH}\|_2\leq C  \|\omega_0-\omega_0^{RH}\|_2,
    \end{equation}
\end{theorem}

In broad terms, our method of proof follows from the intuition put forward by Wirosoetino and Shepherd in \cite{Wirosoetisno1999} using the behavior of Casimir invariants
$$
C_k=\int_{\esse^2} \omega^k\;\mathrm{d}\mathcal{S}.
$$
This strategy was used by Elgindi in \cite{elgindi2023remarkstabilityenergymaximizers}, to study the related problem of stability of steady states for the Euler equation in $\mathbb{T}^2$, see Wang and Zuo \cite{Wang2023} for qualitative stability of the same problem. The work of Elgindi \cite{elgindi2023remarkstabilityenergymaximizers}
gives the Wirosoetino and Shepherd framework a more concise mathematical footing through the use of the inverse function theorem. The main idea is to write the Casimirs explicitly in terms of the active Fourier modes of the solution. In our case $\{c_{2,m}\}_{m=-2}^2$, as the energy does not cascade out of the first and second shell, see \cite{Wirosoetisno1999,Costantin2022} and \cref{lemma:Z_prime} for a precise statement. Hence, the stability question can be reduced to understanding the evolution of the 6 Fourier coefficients that describe all functions belonging to a combination of the first and second shell, see \eqref{eq:2ndshell}.
We can further reduce the degrees of freedom by considering invariance under rotations, see \cref{subsec:reduced_dof}. Stability follows by finding a suitable combination of Casimirs, as many as degrees of freedom the solution, that form a local change of variables. The caveat is that any combination of Casimirs might be degenerate in the sense of deficient rank of the Jacobian of the map defined by the Casimirs. This accounts for the change in the exponent of stability. In fact, the degeneracy in the exponent of stability is already present in the work of Elgindi \cite{elgindi2023remarkstabilityenergymaximizers}.

In terms of extensions of our methods, we note that Rossby--Haurwitz waves can also be constructed in the case when the domain is an ellipsoid, see Xu \cite{xu2023nonzonalrossbyhaurwitzsolutions2d}. In particular, the case of an oblate spheroid provides a natural framework to account for the equatorial bulge observed in the gas giants: Jupiter deviates from spherical symmetry by about 6.5\%, while Saturn shows an even larger deviation of nearly 10\%. Understanding how this affects stability of the Rossby--Haurwitz wave is an open problem, and might provide and explanation to why some zonal flows are observed in these planets. We also mention the case of the non-rectangular torus as a possible application of the tools exposed in this paper, see Wang \cite{wang2025orbitalstabilitylaplacianeigenstates}.

The paper is organized as follows. \cref{sec:preliminaries} introduces the spherical harmonics, recalls some useful properties of the Euler equation on $\esse^2$, and proves two useful classical inverse function Lemmas.
\cref{sec:main} contains the proof of our main results \cref{thm:2} and \cref{thm:main}.

\section{Notations and preliminaries} \label{sec:preliminaries}
In this section, we set up the notation and collect preliminary results for \cref{sec:main}.

\cref{subsec:spherical_harmonics} recalls the definition and properties of spherical harmonics.
\cref{subsec:known_results_euler} collects key properties of Euler flows on $\mathbb{S}^2$, including the main conserved quantities and their consequences. \cref{subsec:calculus_lemmas} proves two classical inverse function lemmas that generalize the 1-d calculus results of Elgindi \cite[Section 2.2]{elgindi2023remarkstabilityenergymaximizers} to higher dimensions.

\subsection{Spherical harmonics} \label{subsec:spherical_harmonics}
We recall that $x\in\mathbb{S}^2$ can be parametrized with $x=(\sin\theta\cos\phi,\sin\theta\sin\phi,\cos\theta)$, using the spherical coordinates $(\theta,\phi)\in[0,\pi]\times[0,2\pi)$. The classical orthonormal basis of $L^2(\mathbb{S}^2)$ is given by the spherical harmonics $\{Y_{l,m}\}\subset\mathcal{C}^\infty(\mathbb{S}^2)$, which are the eigenfunctions of the Laplacian
\begin{equation*}
    \Delta Y_{l,m}=-l(l+1)Y_{l,m}\quad\forall l\in\mathbb{N},m\in\{-l,...,l\}.
\end{equation*}
The spherical harmonics are defined by
\[
Y_{l,m}(\theta, \phi) =
\begin{cases}
(-1)^m \sqrt{2} \sqrt{\frac{2l + 1}{4\pi} \frac{(l - |m|)!}{(l + |m|)!}} P_{l,|m|}(\cos \theta) \sin(|m| \phi), & \text{if } m < 0, \\
\sqrt{\frac{2l + 1}{4\pi}} P_l(\cos \theta), & \text{if } m = 0, \\
(-1)^m \sqrt{2} \sqrt{\frac{2l + 1}{4\pi} \frac{(l - m)!}{(l + m)!}} P_{l,m}(\cos \theta) \cos(m \phi), & \text{if } m > 0,
\end{cases}
\]
where \(\{P_l\}_{l\in\enne} \) are the Legendre polynomials, and \(\{P_{l,m}\}_{l,m>0}\) are the associated Legendre functions.
Explicitly, they are given by
\[
P_l(s) = \frac{1}{2^l l!} \diff{^l}{s^l}\left( s^2 - 1 \right)^l, \quad P_{l,m}(s) = (-1)^m (1 - s^2)^{m/2} \diff{^m}{s^m}P_l(s),\qquad s \in (-1, 1).
\]
Given $n\in\enne$, and denoting the $l$-th energy shell by $\mathcal{H}^l:=\operatorname{Span}\{Y_{l,m}\}_{m=-l}^l$, the projector $\mathbb{P}_n:L^2(\esse^2)\to\mathcal{H}^n$ is defined by
\begin{equation*}
    \mathbb{P}_n\left(\sum_{l\in\enne}\sum_{m=-l}^lc_{l,m}Y_{l,m}\right)=\sum_{m=-n}^nc_{n,m}Y_{n,m}\quad\forall\{c_{l,m}\}\subset\erre.
\end{equation*}
Furthermore, we define
\begin{equation*}
    \mathbb{P}_{\neq n}:=I-\mathbb{P}_n=\sum_{l\neq n}\mathbb{P}_l,\quad\mathbb{P}_{\leq n}:=\sum_{l\leq n}\mathbb{P}_l,\quad\mathbb{P}_{\geq n}:=\sum_{l\geq n}\mathbb{P}_l.
\end{equation*}

\subsection{Known results for Euler} \label{subsec:known_results_euler}
The dynamics generated by the Euler equation on $\esse^2$ \eqref{eq:euler_rotating} satisfy the following conservation laws.
First, the equation is Hamiltonian, with the Hamiltonian given by the kinetic energy
\begin{equation}
    H(\omega)=\frac{1}{2}\int_{\esse^2}|\nabla^\perp \psi|^2\;\mathrm{d}\mathcal{S}=-\frac{1}{2}\int_{\esse^2}\omega\Delta^{-1}\omega\;\mathrm{d}\mathcal{S}=\frac{1}{2}\|\omega\|^2_{H^{-1}},
    \label{eq:energy}
\end{equation}
which is also a conserved quantity. Moreover, the Euler equation on $\esse^2$ \eqref{eq:euler_rotating} preserves the level sets of the vorticity, which implies the conservation of infinitely many quantities, known as Casimirs
\begin{equation}
    C_f(\omega)=\int_{\esse^2}f(\omega)\;\mathrm{d}\mathcal{S}\quad\forall f\in\mathcal{C}^0(\erre).
    \label{eq:Casimirs}
\end{equation}
In particular, every $L^p$ norm of the vorticity is a conserved quantity, as is the $L^2$ norm which is often referred to as \textit{enstrophy}. Finally, due to the spherical symmetry, the angular momentum
\begin{equation}
    L(\omega)=\int_{\esse^2}\omega\vec{\boldsymbol{n}}\;\mathrm{d}\mathcal{S},
    \label{eq:cons_ang_mom}
\end{equation}
where $\vec{\boldsymbol{n}}=(x_1,x_2,x_3)$, is a conserved quantity. To show the conservation of angular momentum, we test the vorticity equation in \eqref{eq:euler_rotating} against the coordinate function $x_\alpha$ and integrate by parts. Using that $\partial_t \omega = - \nabla \omega \cdot \nabla^\perp \psi$ and that $\nabla \cdot \nabla^\perp \psi = 0$, we obtain
\begin{equation}
    \frac{d}{dt}L_\alpha(\omega) 
    = -\int_{\mathbb{S}^2} \nabla \omega \cdot \nabla^\perp \psi \, x_\alpha \; \mathrm{d}\mathcal{S} 
    = -\int_{\mathbb{S}^2} \nabla^\perp x_\alpha \cdot \nabla \psi \,\Delta \psi \; \mathrm{d}\mathcal{S},
    \qquad \alpha \in \{1,2,3\}.
    \label{eq:step_pf_cons_L_clean}
\end{equation}
The first-order operator
\[
T_\alpha : \psi \mapsto \nabla^\perp x_\alpha \cdot \nabla \psi
\]
is skew-adjoint and commutes with the Laplacian. Combining these, and recalling that $\Delta$ is self-adjoint, the right-hand side of \eqref{eq:step_pf_cons_L_clean} vanishes:
\[
\frac{d}{dt}L_\alpha(\omega)=-\int_{\mathbb{S}^2} \nabla^\perp x_\alpha \cdot \nabla \psi \, \Delta \psi \; \mathrm{d}\mathcal{S}
= -\int_{\mathbb{S}^2} (T_\alpha \psi)(\Delta \psi) \; \mathrm{d}\mathcal{S}
= 0, \qquad \alpha \in \{1,2,3\},
\]
which shows that the angular momentum \eqref{eq:cons_ang_mom} is conserved. 

In particular, the conservation of the angular momentum implies that the first energy shell $\mathcal{H}^1$ is frozen during the evolution
$$
\int_{\esse^2}\omega_0 Y_{1,m}\;\mathrm{d}\mathcal{S}=\int_{\esse^2}\omega_t Y_{1,m}\;\mathrm{d}\mathcal{S},\quad\forall t\geq0,\ m=-1,0,1,
$$
as, up to a constant and re-labelling, $\{Y_{1,m}\}_{m=-1}^1$ equal to $\{x_i\}_{i=1}^3$. Since the conservation of angular momentum freezes the modes in the first shell, it is natural to wonder what happens in the second shell and beyond. To this regard, we state and prove the following lemma based on the arguments of \cite[Section 3]{Wirosoetisno1999}, see also \cite[Theorem 7 (iii)]{constantin2025onsetinstabilityzonalstratospheric}.
\begin{lemma} \label{lemma:Z_prime}
    Let $\omega_t$ be a solution to Euler equation in $\esse^2$ \eqref{eq:euler_rotating} with initial datum $\omega_0\in L^2(\esse^2)$, then for all $t\geq0$ it holds
    \begin{equation}
        \frac{1}{2}\|\mathbb{P}_{\geq3}\omega_0\|_2^2\leq \|\mathbb{P}_{\geq3}\omega_t\|_2^2\leq2\|\mathbb{P}_{\geq3}\omega_0\|_2^2.
        \label{eq:Zprime_eq_ofthelemma}
    \end{equation}
\end{lemma}
\begin{proof}
    Using the conservation of the $L^2$ norm and of the angular momentum, we have that 
    \begin{equation}\label{eq:Zprime_proof_step-1}
        \|\mathbb{P}_2\omega_t\|_2^2+\|\mathbb{P}_{\geq3}\omega_t\|_2^2=\|\mathbb{P}_2\omega_0\|_2^2+\|\mathbb{P}_{\geq3}\omega_0\|_2^2
    \end{equation}
    is a conserved quantity.
    Similarly, using the conservation of the energy and of the angular momentum, we have that
    \begin{equation}\label{eq:Zprime_proof_step0}
        \|\mathbb{P}_2\omega_t\|^2_{H^{-1}}+\|\mathbb{P}_{\geq3}\omega_t\|^2_{H^{-1}}=\|\mathbb{P}_2\omega_0\|^2_{H^{-1}}+\|\mathbb{P}_{\geq3}\omega_0\|^2_{H^{-1}}
    \end{equation}
    is also a conserved quantity.
    Using the identity
    \begin{equation*}
        \|\mathbb{P}_2\omega\|^2_{H^{-1}}=\frac{1}{6}\|\mathbb{P}_2\omega\|_2^2,
    \end{equation*}
    that holds for any function $\omega\in L^2(\esse^2)$, we can re-write \eqref{eq:Zprime_proof_step0} as
    \begin{equation}\label{eq:Zprime_proof_step1}
        \|\mathbb{P}_2\omega_t\|_2^2+6\|\mathbb{P}_{\geq3}\omega_t\|^2_{H^{-1}}=\|\mathbb{P}_2\omega_0\|_2^2+6\|\mathbb{P}_{\geq3}\omega_0\|^2_{H^{-1}}.
    \end{equation}
    Combining \eqref{eq:Zprime_proof_step-1} and \eqref{eq:Zprime_proof_step1}, we obtain
    \begin{equation}\label{eq:Zprime_proof_step2}
        \|\mathbb{P}_{\geq3}\omega_t\|_2^2-6\|\mathbb{P}_{\geq3}\omega_t\|^2_{H^{-1}}=\|\mathbb{P}_{\geq3}\omega_0\|_2^2-6\|\mathbb{P}_{\geq3}\omega_0\|^2_{H^{-1}}.
    \end{equation}
    For the $l\geq3$ modes we have the Poincaré inequality for any function $\omega\in L^2(\esse^2)$:
    \begin{equation}\label{eq:poincare_inequality}
        \frac{1}{2}\|\mathbb{P}_{\geq3}\omega\|_2^2 \le \|\mathbb{P}_{\geq3}\omega\|_2^2-6\|\mathbb{P}_{\geq3}\omega\|^2_{H^{-1}}.
    \end{equation}
    Applying \eqref{eq:poincare_inequality} on $\omega_t$ and $\omega_0$ within the identity \eqref{eq:Zprime_proof_step2}, we obtain the desired result \eqref{eq:Zprime_eq_ofthelemma}.
\end{proof}
As a consequence of Lemma \ref{lemma:Z_prime}, given an initial condition $\omega_0\in L^2(\esse^2)$ satisfying $\|\mathbb{P}_{\geq3}\omega_0\|_2\leq\varepsilon$ for some $\varepsilon>0$, it immediately follows that
\begin{equation}
    \|\mathbb{P}_{\geq3}\omega_t\|_2\leq\sqrt{2}\varepsilon\quad\forall t\geq0.
    \label{eq:stab_lgeq3}
\end{equation}
Along with the conservation of angular momentum \eqref{eq:cons_ang_mom}, the estimate \eqref{eq:stab_lgeq3} shows that energy cannot transfer between the second shell and higher shells. However, the coefficients of the five second-shell modes may still evolve in time. For non-zonal steady states in the second shell, small perturbations can lead to $L^2$ instability; see \cite{constantin2025onsetinstabilityzonalstratospheric}. Therefore, in this manuscript we focus on proving orbital stability.

\subsection{\texorpdfstring{Inverse function lemmas}{Inverse function lemmas}} \label{subsec:calculus_lemmas}
Next, we recall classical stability estimates associated to inverse functions of smooth maps $F:\R^n\to\R^n$, which will be applied in our main results. We start by considering the case when the gradient of $\nabla F(x_0)$ is invertible, and show that the inverse function is Lipschitz continuous. 
\begin{lemma} \label{lemma:grad_neq0}
        Let $F:\erre^n \to\erre^n$ be a smooth function and $x_0\in \R^n$. Assume that $\det\nabla F(x_0)\neq0$, then there exists $\delta>0$ and $C=C(F,x_0)>0$ such that for all $x\in B_\delta(x_0)$, we have the stability estimate
        \begin{equation}\label{eq:grad_neq0}
            |x-x_0|\leq C |F(x)-F(x_0)|.
        \end{equation}
\end{lemma}
When $\nabla F(x_0)$ is not invertible, several different cases may occur depending on the dimension of $\ker\nabla F(x_0)$ and its interaction with the second derivatives $\nabla^2F(x_0)$. In the following Lemma, we consider the specific case when $\dim\ker\nabla F(x_0)=1$ and $\nabla^2F(x_0)(e,e)\notin \operatorname{col}\nabla F(x_0)$, where $e\in\ker\nabla F(x_0)$ is a unit vector. The classical literature refers to such $x_0$ as a fold point. Under these assumptions, there exists an invertible smooth change of variables $x\in B_\delta(x_0)\mapsto y\in B_\delta(0)$ such that in this new variable
    $$
        F(x(y))-F(x_0)=(y_1,y_2,\ldots,y^2_n),
    $$
see for instance \cite[Chapter III, Theorem 4.5]{golubitsky1973stable}. In this setting, we expect a different stability estimate than the one in \cref{lemma:grad_neq0}.
In fact, the inverse does not even exist, however we have that any of its two branches are $\mathcal{C}^{1/2}$, like the positive and negative square root near zero.
\begin{lemma} \label{lemma:ker_1dim}
    Let $F:\erre^n\to\erre^n$ be a smooth function and $x_0\in \R^n$. Assume that
    \begin{equation*}
        \dim\ker\nabla F(x_0)=1,
    \end{equation*}
    and that
    \begin{equation}
        \nabla^2F(x_0)(e,e)\notin \operatorname{col}\nabla F(x_0),
        \label{eq:1dimker_crucial_assumption}
    \end{equation}
    where $e\in\ker\nabla F(x_0)$ is a unit vector. Then, there exists $\delta>0$ and a constant $C=C(F,x_0)>0$, such that for all $x\in B_\delta(x_0)$ we have the stability estimate
    \begin{equation}\label{eq:ker_1dim}
        |x-x_0|^2\leq C|F(x)-F(x_0)|.
    \end{equation}
\end{lemma}

\subsubsection{Proof of \cref{lemma:grad_neq0}, \cref{lemma:ker_1dim}}
To begin with, let us introduce the notation and recall some useful calculus properties. For $k\in\enne\setminus\{0\}$, and for any $\xi\in\erre^n$ we define the multi-linear operator $\nabla^kF(\xi):(\erre^{n})^k\to\erre^n$ by
\begin{equation*}
    [\nabla^kF(\xi)(v_1,\dots,v_k)]_i=\sum_{i_1,\dots,i_k=1}^n(\partial_{i_1}\cdot\cdot\cdot\partial_{i_k}F_i(\xi))\cdot[v_1]_{i_1}\cdot\cdot\cdot[v_k]_{i_k}.
\end{equation*}
This tensor notation allows to write the Taylor expansion of $F$ about $x_0$ truncated at order $N\in\enne\setminus\{0\}$ in the following form. Given $\delta>0$ and $x\in B_\delta(x_0)$, there exists $\xi\in B_\delta(x_0)$ such that
\begin{equation}
    F(x)-F(x_0)=\sum_{k=1}^{N-1}\frac{1}{k!}\nabla^kF(x_0)(x-x_0,\dots,x-x_0)+\frac{1}{N!}\nabla^NF(\xi)(x-x_0,\dots,x-x_0). \label{eq:taylor_expansion}
\end{equation}
For $k\in\enne$, we define the operator norm of $\nabla^k F(\xi)$ as 
\begin{equation*}
    \|\nabla^k F(\xi)\|:=\sup\left\{\nabla^kF(\xi)(v_1,\dots,v_k)\,\middle\vert\, v_j\in\erre^n,\ \|v_j\|=1,\ j=1,\dots,k\right\}.
\end{equation*}

\begin{proof}[Proof of \cref{lemma:grad_neq0}]
    We consider $\delta\in (0,1)$, and apply Taylor's theorem to the function $F$ at the point $x_0$ to obtain that for any $x\in B_\delta(x_0)$ there exists $\xi\in B_\delta(x_0)$ such that
    \begin{equation*}
        \nabla F(x_0)(x-x_0)+\frac{1}{2}\nabla^2F(\xi)(x-x_0,x-x_0)=F(x)-F(x_0).
    \end{equation*}
    Taking norms and using the triangle inequality, we have 
    \begin{equation}\label{eq:grad_neq0_proof_step0}
        \left|\nabla F(x_0)(x-x_0)\right|-\frac{1}{2}\left|\nabla^2F(\xi)(x-x_0,x-x_0)\right|\le |F(x)-F(x_0)|.
    \end{equation}
    Using that $\nabla F(x_0)$ is an invertible matrix, we can use the operator norm of inverse matrix to estimate
    \begin{equation*}
       \|x-x_0\|= \left\|(\nabla F(x_0))^{-1}\nabla F(x_0)(x-x_0)\right\|\le \|(\nabla F(x_0))^{-1}\|\|\nabla F(x_0)(x-x_0)\|.
    \end{equation*}
    Replacing the previous bound in \eqref{eq:grad_neq0_proof_step0} and taking the supremum over $B_1(x_0)$ of the operator norm of $\nabla^2F(\xi)$ we have the bound
    \begin{equation}\label{eq:grad_neq0_proof_step1}
        \frac{1}{\|\nabla F(x_0)^{-1}\|}|x-x_0|-\frac{1}{2}\sup_{\xi\in B_1(x_0)}\|\nabla^2 F(\xi)\||x-x_0|^2\le |F(x)-F(x_0)|.
    \end{equation}
    Finally, choosing 
    $$
        \delta=\frac{1}{\|\nabla F(x_0)^{-1}\|\sup_{\xi\in B_1(x_0)}\|\nabla^2 F(\xi)\|},
    $$
    we have that for any $x\in B_\delta(x_0)$, the following estimate holds
    $$
        \frac{1}{2\|\nabla F(x_0)^{-1}\|}|x-x_0|\le \frac{1}{\|\nabla F(x_0)^{-1}\|}|x-x_0|-\frac{1}{2}\sup_{\xi\in B_1(x_0)}\|\nabla^2 F(\xi)\||x-x_0|^2.
    $$
    Combining this estimate with \eqref{eq:grad_neq0_proof_step1}, we get \eqref{eq:grad_neq0} with $C=2\|\nabla F(x_0)^{-1}\|$.
\end{proof}

\begin{proof}[Proof of \cref{lemma:ker_1dim}]
    We define
    \begin{equation}
        \lambda:=\inf_{w\in\ker\nabla F(x_0)^\perp}|\nabla F(x_0)w|>0,
    \end{equation}
    and we pick 
    $$
        \gamma=\min\left\{\frac{\lambda}{4\sup_{\xi\in B_1(x_0)}\|\nabla^2 F(\xi)\|},1\right\}.
    $$
    We take $\delta\in(0,\gamma)$ to be chosen later. For any $x\in B_\delta(x_0)$, we decompose $x=x_0+te+sw$ for some $s,t\in[0,\delta)$, with $e\in\ker\nabla F(x_0),\,|e|=1$ and $w\in\ker\nabla F(x_0)^\perp,\,|w|=1$.
    We consider the two cases $\frac{t^2}{\gamma}\leq s$ and $s<\frac{t^2}{\gamma}$.
    
    \textbf{First case.} If $\frac{t^2}{\gamma}\leq s$, the proof goes similarly to the one of \cref{lemma:grad_neq0}.
    Using a Taylor expansion \eqref{eq:taylor_expansion} of order two, taking norms and using the fact that $\nabla F(x_0)e=0$, we have
    \begin{equation*}
        \lambda s-\frac{1}{2}\sup_{\xi\in B_1(x_0)}\|\nabla^2 F(\xi)\| |x-x_0|^2\le\left|\nabla F(x_0)w s+\frac{1}{2}\nabla^2 F(\xi)(x-x_0,x-x_0)\right| = |F(x)-F(x_0)|.
    \end{equation*}
    Using the triangle inequality, and the assumption $\frac{t^2}{\gamma}\leq s$, we have
    \begin{equation*}
        |x-x_0|^2=|sw+te|^2\le 2s^2+2t^2\leq 2(\delta+\gamma) s\leq 4\gamma s,
    \end{equation*}
    where in the last inequality we have used that $\delta<\gamma$. Combining the two previous estimates, we obtain
    \begin{equation*}
        \frac{\lambda}{8\gamma} |x-x_0|^2\le\frac{\lambda}{2} s\le \left(\lambda-2\delta \sup_{\xi\in B_1(x_0)}\|\nabla^2 F(\xi)\|\right) s \le |F(x)-F(x_0)|,
    \end{equation*}
    where we have used the definition of $\gamma$. This shows the desired estimate holds with $C=\frac{8\gamma}{\lambda}$.
    
    \textbf{Second case.} We now consider $s<\frac{t^2}{\gamma}$. We start with the Taylor expansion \eqref{eq:taylor_expansion} of order three, and we take norms to get
    \begin{eqnarray*}
        \underbrace{\left|\nabla F(x_0) w s+\frac{1}{2}\nabla^2 F(x_0)(e,e)t^2+\nabla^2 F(x_0)(e,w)st+\nabla^2 F(x_0)(w,w)t^2+\frac{1}{6}\nabla^3F(\xi)(x-x_0)^3\right|}_{I}\le |F(x)-F(x_0)|.
    \end{eqnarray*}
    Using the triangle inequality, we can estimate
    $$
        \left|\nabla F(x_0) w s+\frac{1}{2}\nabla^2 F(x_0)(e,e)t^2\right|-\sup_{\xi\in B_1(x_0)}\|\nabla^2 F(x_0)\|\left(st+\frac{s^2}{2}\right)-\frac{1}{6}\sup_{\xi\in B_1(x_0)}\|\nabla^3F(\xi)\|(t^3+s^3+3s^2t+3st^2)\le I.
    $$
    We use the assumption of non-degeneracy \eqref{eq:1dimker_crucial_assumption}, and define
    \begin{equation*}
        0<\eta:=|\mathbb{P}_{\operatorname{col}\nabla F(x_0)^\perp}\nabla^2F(x_0)(e,e)|,
    \end{equation*}
    to substitute the first term in the previous inequality, obtaining
    \begin{align*}
        \frac{1}{2}\eta t^2-\sup_{\xi\in B_1(x_0)}\|\nabla^2 F(x_0)\|\left(st+\frac{s^2}{2}\right)-\frac{1}{6}\sup_{\xi\in B_1(x_0)}\nabla^3F(\xi)(t^3+s^3+3s^2t+3st^2)\leq |F(x)-F(x_0)|.
    \end{align*}
    We now use the fact that $s<\frac{t^2}{\gamma}$ and $s,t<\delta<\gamma$, to obtain
    $$
        \left(\frac{1}{2}\eta -\sup_{\xi\in B_1(x_0)}\|\nabla^2 F(x_0)\|\frac{3}{2}\frac{\delta}{\gamma}-\frac{4}{3}\sup_{\xi\in B_1(x_0)}\|\nabla^3F(\xi)\|\delta \right)t^2\leq |F(x)-F(x_0)|.
    $$
    We restrict
    $$
        \delta<\frac{\eta}{4\left(\sup_{\xi\in B_1(x_0)}\|\nabla^2 F(x_0)\|\frac{3}{2\gamma}+\frac{4}{3}\sup_{\xi\in B_1(x_0)}\|\nabla^3F(\xi)\|\right)},
    $$
    to obtain, for all $x\in B_\delta(x_0)$, the estimate
    \begin{equation*}
        \left(3\frac{\delta}{\gamma}+1\right)^{-1}|x-x_0|^2 \le t^2\le \frac{4}{\eta}|F(x)-F(x_0)|.
    \end{equation*}
    Restricting further $\delta<\gamma/3$, we can conclude that the desired estimate \eqref{eq:ker_1dim} holds with $C=\frac{8}{\eta}$.
\end{proof}

\section{Stability of RH waves} \label{sec:main}
Following the ideas of Wirosoetisno-Shepherd \cite{Wirosoetisno1999}, the strategy of all our stability proofs is to first reduce the problem to a finite dimensional question by using that energy does not cascade out of the first two shells. Without loss of generality, for simplicity from now on we assume that the rotation axis $p=e_3$ and the normalization $\|\mathbb{P}_2 \omega_0^{RH}\|_2=1$.
The normalization can be always achieved by noticing the quadratic nature of the Euler equation \eqref{eq:euler_rotating}. Namely, scaling the initial condition by a factor is equivalent to performing a rescaling in time of the equation. Next, we show that the solution to the Euler equation \eqref{eq:euler_rotating} remains close to a solution that is projected onto the first two shells.
\begin{lemma}\label{lem:projection1}
There exist constants $\bar\varepsilon > 0$ and $c_*>0$ such that the following holds. Let
\[
    \omega_0^{RH} = \Omega x_3 + Y_2,
\]
where $\Omega \in \mathbb{R}$ and $Y_2 \in \mathcal{H}^2$ is a Laplacian eigenfunction in the second shell, normalized so that $\|Y_2\|_2 = 1$. For any initial datum $\omega_0 \in L^2(\mathbb{S}^2)$ satisfying
\[
    \|\omega_0 - \omega_0^{RH}\|_2 < \bar\varepsilon,
       \]
let $\omega_t$ be the solution to the Euler equation with initial condition $\omega_0$. Then, for all $t \geq 0$,
\[
    \|\omega_t - \tilde\omega_t\|_2 \leq c_* \|\omega_0 - \omega_0^{RH}\|_2,
\]
where
\[
    \tilde\omega_t = \Omega x_3 + \frac{\mathbb{P}_2 \omega_t}{\|\mathbb{P}_2 \omega_t\|_2}.
\]
\end{lemma}
\begin{proof}[Proof of \cref{lem:projection1}]
For any perturbation satisfying $\|\omega_0-\omega_0^{RH}\|_2\leq\varepsilon$, we readily obtain that
$$
\|\mathbb{P}_1\omega_0-\Omega x_3 \|_2\le \eps,\quad \|\mathbb{P}_{\ge 3}\omega_0\|_2\leq\eps,\quad\text{and}\quad |\|\mathbb{P}_2\omega_0\|_2-1|\leq\eps.
$$
The conservation of the angular momentum and \cref{lemma:Z_prime} imply that similar bounds hold for $\omega_t$, the solution to \eqref{eq:euler_rotating}, for all $t\geq0$. Namely, we have that
$$
\|\mathbb{P}_1\omega_t-\Omega x_3 \|_2\le \eps,\quad \|\mathbb{P}_{\ge 3}\omega_0\|_2\leq 2\eps,\quad\text{and}\quad |\|\mathbb{P}_2\omega_0\|_2-1|\leq 2\eps.
$$
Hence, we can conclude that if $\|\omega_0-\omega_0^{RH}\|_2\leq\varepsilon$ for $\varepsilon>0$ small enough, then for any $t\geq0$ it holds
$$
\left\|\omega_t- \frac{\mathbb{P}_2\omega_t}{\|\mathbb{P}_2\omega_t\|_2}-\Omega x_3 \right\|_2\le c\eps.
$$ 
This concludes the proof of \cref{lem:projection1}.
\end{proof}
Now that we have reduced the problem to the evolution of the Fourier coefficients in the second shell
\begin{equation}
    \tilde\omega_t = \Omega x_3 + \sum_{m=-2}^2 \tilde\omega_{2,m}(t) Y_{2,m},\qquad\mbox{and}\qquad \omega_0^{RH} = \Omega x_3 + \sum_{m=-2}^2 \omega^{RH}_{2,m} Y_{2,m},
    \label{eq:projected_solution}
\end{equation}
with the normalization condition on the Fourier coefficients
\begin{equation}
    \sum_{m=-2}^2 |\tilde\omega_{2,m}(t)|^2 = 1= \sum_{m=-2}^2 |\omega^{RH}_{2,m}|^2,
    \qquad \forall t \geq 0,
    \label{eq:normalization_condition}
\end{equation}
we can proceed to study the stability of the RH wave by considering the behaviour of the Casimirs. The following lemma shows that the Casimirs of the projected solution $\tilde\omega_t$ remain close to the Casimirs of the RH wave $\omega_0^{RH}$. For simplicity, we focus on understanding the behaviour of the polynomial Casimirs
    \begin{equation}
        C_k(\omega):=\int_{\mathbb{S}^2}\omega^k\;\mathrm{d}\mathcal{S},\quad k\in\enne.
        \label{eq:PolCasimirs} 
    \end{equation}

\begin{lemma}\label{lem:projection2}
    Under the same assumptions and notation as in \cref{lem:projection1}, for each $k \in \mathbb{N}$ there exists a constant $c_k > 0$ such that for all $t \geq 0$,
    \[
        \left| C_k(\tilde\omega_t) - C_k(\omega_0^{RH}) \right| \leq c_k \|\omega_0 - \omega_0^{RH}\|_2.
    \]
\end{lemma}
\begin{proof}[Proof of \cref{lem:projection2}]
We start by introducing a useful modification of the Casimirs. Let $\chi\in\mathcal{C}^\infty(\erre)$ be even and equal to $1$ on $[0,1]$, and to $0$ on $[2,\infty)$. Then, consider the modified Casimirs
\begin{equation}
    \tilde{C}_k(\omega):=\int_{\mathbb{S}^2}\omega^k\chi\biggl(\frac{\omega}{|\Omega|+(1+\sqrt{3})\sqrt{5/\pi}}\biggr)\;\mathrm{d}\mathcal{S},\quad k\in\enne,
    \label{eq:modif_Casimirs}
\end{equation}
satisfy a Lipschitz stability condition with respect to $L^2$
\begin{equation}
    |\tilde{C}_k(\omega_1)-\tilde{C}_k(\omega_2)|\leq c_k \|\omega_1-\omega_2\|_2.
    \label{eq:Lips_of_casimirs}
\end{equation}
For $\omega_t$ the solution to \eqref{eq:euler_rotating}, the modified Casimirs are conserved quantities. Thus, we have
$$
\tilde{C}_k(\omega_t)=\tilde{C}_k(\omega_0), \qquad   \forall t\geq0.
$$
Hence, by \eqref{eq:Lips_of_casimirs}, the triangle inequality, and \cref{lem:projection1}, we get that there exists a constant $c_k>0$ such that
$$
|\tilde{C}_k(\tilde{\omega}_t)-\tilde{C}_k(\omega^{RH}_0)|\le |\tilde{C}_k(\tilde{\omega}_t)-\tilde{C}_k(\omega_t)|+|\tilde{C}_k(\omega_t)-\tilde{C}_k(\omega_0)| +|\tilde{C}_k(\omega_0)-\tilde{C}_k(\omega^{RH}_0)|\leq c_k\varepsilon,\quad\forall t\geq0.
$$
Finally, we observe that $\|\tilde\omega_t\|_{L^\infty}\leq |\Omega|+(1+\sqrt{3})\sqrt{5/\pi}$ for any $t\geq0$, and the same holds for $\omega_0^{RH}$. Thus, the cut-off function $\chi$ in \eqref{eq:modif_Casimirs} is always equal to $1$ when evaluated at $\tilde\omega_t$ and $\omega_0^{RH}$. This implies that
$$\tilde{C}_k(\tilde{\omega}_t)=C_k(\tilde{\omega}_t)\quad\text{and}\quad \tilde{C}_k(\omega_0^{RH})=C_k(\omega_0^{RH}),$$
and the proof is concluded.
\end{proof}

Now the game is to create a non-degenerate change of variables map directly on the Fourier coefficients $\tilde\omega_{2,m}(t)$, for $m=-2,-1,0,1,2$. For $k\in\enne$, we define the functions $F_k:\mathbb{R}^5 \to \mathbb{R}$ by
\[
    F_k(\tilde\omega_{2,-2}, \tilde\omega_{2,-1}, \tilde\omega_{2,0}, \tilde\omega_{2,1}, \tilde\omega_{2,2}) = C_k\left(\Omega x_3+\sum_{m=-2}^2 \tilde\omega_{2,m} Y_{2,m}\right).
\]
Now the goal is to show that a map like $F = (F_2, F_3, F_4, F_5): \mathbb{S}^4 \to \mathbb{R}^4$ is non-degenerate at the point $(\tilde\omega_{2,m}^{RH})_{m=-2}^2\in \mathbb{S}^4$ corresponding to the RH wave $\omega_0^{RH}$. The restriction to $\mathbb{S}^4$ follows from the normalization \cref{eq:normalization_condition}. If the map is not degenerate, small variations in the values of the Casimirs $C_k(\tilde\omega_t)$ will then correspond to small variations in the Fourier coefficients $(\tilde\omega_{2,m}(t))_{m=-2}^2$, which will allow us to conclude stability, by applying \cref{lem:projection1} and \cref{lem:projection2}.

However, as noted in the introduction \eqref{eq:degeneracy_case}, we should notice that the stability results are set only up to the appropriate rotations. Hence, we can not expect to find a non-degenerate map $F$ as written above, unless we further reduce the degrees of freedom by using the rotational invariance of the problem. This will be done in the next section, for the different cases. We start with the case $\Omega=0$, as after rotations we are left with a single degree of freedom in the second shell, which can be easily handled like in the paper of Elgindi~\cite{elgindi2023remarkstabilityenergymaximizers} for $\mathbb{T}^2$.

\subsection{Proof of \cref{thm:2}} \label{sec:thm2}
In the case $\Omega=0$, we have that small perturbations can induce a rotation in any direction. Rotations through any axis have 3 degrees of freedom, 2 for the axis of rotation and 1 for the rotation angle.
The second shell has 5 degrees of freedom, which means that we can reduce the problem to a 2-dimensional one. Finally, if we use the normalization \eqref{eq:normalization_condition} we are left with a single degree of freedom.
\begin{lemma}
    Let $\omega_0^{RH} = Y_2 \in \mathcal{H}^2$ be a Laplacian eigenfunction in the second shell, normalized so that $\|Y_2\|_2 = 1$. There exists a rotation $R \in SO(3)$ such that the rotated vorticity can be decomposed as
    $$
        \omega_0^{RH}(Rx)=A Y_{2,0} + B Y_{2,2},
    $$ 
    with $A^2 + B^2 = 1$ and $3/7\le A^2 $.
    \label{lemma:rotation_case_Omega0}
\end{lemma}
We defer the proof of \cref{lemma:rotation_case_Omega0} to the end of this section.
We can now proceed to the proof of stability in the case $\Omega=0$.

\begin{proof}[Proof of \cref{thm:2}]
Using the representation of \cref{lemma:rotation_case_Omega0}, we have that 
$$
\omega_0^{RH}(R_0x) = A^{RH} Y_{2,0} + B^{RH} Y_{2,2},
$$
with $|A^{RH}|^2 + |B^{RH}|^2 = 1$ and $3/7 \leq |A^{RH}|^2$. Using the notation of \cref{lem:projection1}, the projection $\tilde\omega_t$ of the solution can also be written as
\begin{equation*}
    \tilde\omega_t(R(t)x) = A(t) Y_{2,0} + B(t) Y_{2,2},
\end{equation*}
with $A(t)^2 + B(t)^2 = 1$  and $3/7\le A^2(t)$ for all $t \geq 0$. The goal is now to show that the Casimirs $C_k(\tilde\omega_t)$ can be directly expressed in terms of the functions $A(t)$ and $B(t)$. For a general vorticity in the second shell with the given representation we compute the Casimir $C_3$ explicitly as follows:
    \begin{equation}
        \int_{\mathbb{S}^2}(AY_{2,0}+BY_{2,2})^3\;\mathrm{d}\mathcal{S}=\frac{1}{7}\sqrt{\frac{5}{\pi}}A(A^2-3B^2).
        \label{eq:choice_of_k0}
    \end{equation}
Using the normalization we can substitute $B^2=1-A^2$, producing the following scalar function $F:\mathbb{R}\to\mathbb{R}$ of a single real variable
    \begin{equation*}
        F(A)=\frac{1}{7}\sqrt{\frac{5}{\pi}}A(4A^2-3).
    \end{equation*}
By \cref{lem:projection2} and the fact that Casimirs are invariant under rotations, we have that for all $t\geq0$ the following estimate holds
    \begin{equation*}
    |F(A(t))-F(A^{RH})|=|C_3(\tilde\omega_t(R(t)x))-C_3(\omega^{RH}(R_0x))|=|C_3(\tilde\omega_t(x))-C_3(\omega^{RH}(x))|\leq c_3\varepsilon,
    \end{equation*}
As long as $F'(A^{RH})\ne0$, we can use the inverse function theorem \cref{lemma:grad_neq0}, to show that there exist constants $C_1>0$ and $\eps_0>0$ such that if $|A-A^{RH}|<\eps_0$, then we have the stability estimate
$$
    |A-A^{RH}|\leq C|F(A)-F(A^{RH})|.
$$
Under the assumption that the initial perturbation is sufficiently small $0<\bar\eps<\frac{1}{C}\eps_0$, for some $C$ large enough, we can assert that
$$
|A(0)-A^{RH}|\le \|\omega_0-\omega_0^{RH}\|_2<\bar\eps\le \eps_0.
$$
By continuity of the function $A(t)$, we can conclude that $|A(t)-A^{RH}|<\bar\eps$ for $t$ small enough. Hence, we have that for $t$ small enough the following estimate holds
    $$
        \frac{1}{C}\|\tilde\omega_t(R(t)x)-\omega^{RH}(R_0x) \|_2\leq |A(t)-A^{RH}|\leq C_1|F(A(t))-F(A^{RH})|<C c_3 \|\omega_0 - \omega_0^{RH}\|_2\le C \bar\eps,
    $$
    where we have used \cref{lem:projection2} in the last inequality. Hence, picking $\bar\eps<\frac{1}{C}\eps_0$ for $C$ large enough, we can conclude that $|A(t)-A^{RH}|<\eps_0$ for all $t\geq0$, and we can apply the inverse function theorem for all times.

    Concluding, as long as 
    $$
        F'(A^{RH})\ne 0,
    $$ 
    we have the desired stability estimate
    $$
        \inf_{R\in SO(3)}\|\tilde\omega_t(Rx)-\omega^{RH}(x)\|_2\leq C\|\omega_0 - \omega_0^{RH}\|_2 .
    $$
    Finally, to check the derivative of $F$ at the point $A^{RH}$, we compute
    \begin{equation*}
        F'(A)=\frac{3}{7}\sqrt{\frac{5}{\pi}}(4A^2-1)
    \end{equation*}
    vanishes only at $A^2=1/4$. Since $A^{RH}$ satisfies $|A^{RH}|^2\ge 3/7>1/4$, we have that $F'(A^{RH})\ne0$. This concludes the proof of stability for the case $\Omega=0$.
\end{proof}

\begin{proof}[Proof of \cref{lemma:rotation_case_Omega0}]
Given any vorticity function purely in the second shell $\omega\in\mathcal{H}^2$:
$$
\omega=\sum_{m=-2}^2\omega_{2,m} Y_{2,m},
$$
there exists a symmetric traceless matrix $Q\in\erre^{3\times3}$ such that
\begin{equation}
    \omega(x)=x^TQx.
    \label{eq:quadratic_form}
\end{equation}
Here, $x\in\esse^2$ are the coordinates of a point on the embedded sphere $\mathbb{S}^2\subset\mathbb{R}^3$. Explicitly, the matrix $Q$ is given by
\begin{equation}
Q=\sqrt{\frac{15}{16\pi}}\begin{bmatrix}
2\omega_{2,2}-\frac{\sqrt{3}}{3}\omega_{2,0} & \omega_{2,-2} & \omega_{2,1} \\
\omega_{2,-2} & -2\omega_{2,2}-\frac{\sqrt{3}}{3}\omega_{2,0} & \omega_{2,-1} \\
\omega_{2,1} & \omega_{2,-1} & \frac{2\sqrt{3}}{3}\omega_{2,0}
\end{bmatrix}.
\label{eq:matrixQ_expl}
\end{equation}
Since $Q$ is symmetric, there exists an orthogonal matrix $R\in\erre^{3\times3}$ such that $R^TQR=diag(\lambda_1,\lambda_2,\lambda_3)$ is diagonal. Up to re-arranging the columns of $R$, we can assume that $R\in\mathbb{SO}(3)$, that is to say, $R$ is a rotation matrix satisfying
\begin{equation*}
    R^T R=I\quad\text{and}\quad\det R=1.
\end{equation*}
Moreover, we can choose the order of the collumns such that both $R\in SO(3)$ and the eigenvalues of $R^TQR$ satisfy the following ordering condition
\begin{equation}
    \max(|\lambda_1|,|\lambda_2|)\leq|\lambda_3|.
    \label{eq:lambda3_max}
\end{equation}

Now, rotating the sphere $\esse^2$ according to the matrix $R$ corresponds to
\begin{equation*}
    \omega(Rx):=x^T(R^TQR)x=\lambda_1x_1^2+\lambda_2x_2^2+\lambda_3x_3^2,
\end{equation*}
where $\lambda_1,\,\lambda_2,\,\lambda_3\in\erre$ are the eigenvalues of $Q$. Since the trace is invariant with similarity transformations, we also have $\lambda_1+\lambda_2+\lambda_3=0$. Finally, recalling that
\begin{equation*}
    Y_{2,0}=\sqrt{\frac{5}{16\pi}}(3x_3^2-1),\quad Y_{2,2}=\sqrt{\frac{15}{16\pi}}(x_1^2-x_2^2),
\end{equation*}
we can show after some algebra that
\begin{equation}
    \omega(Rx)=\underbrace{\frac{\lambda_3}{2}\sqrt{\frac{16\pi}{5}}}_{A}Y_{2,0}+\underbrace{\frac{\lambda_1-\lambda_2}{2}\sqrt{\frac{16\pi}{15}}}_{B}Y_{2,2}.
    \label{eq:relation_AB_lambda}
\end{equation}
By \eqref{eq:lambda3_max}, we have that $|\lambda_1-\lambda_2|\le 2\lambda_3$.
Moreover, since 
\begin{eqnarray*}
    1=\|\omega\|^2&=&A^2+B^2\\
    &=&\left|\frac{\lambda_3}{2}\sqrt{\frac{16\pi}{5}}\right|^2+\left|\frac{\lambda_1-\lambda_2}{2}\sqrt{\frac{16\pi}{15}}\right|^2\\
    &\le&\left|\frac{\lambda_3}{2}\sqrt{\frac{16\pi}{5}}\right|^2+\left|\lambda_3\sqrt{\frac{16\pi}{15}}\right|^2=\left(1+\frac{4}{3}\right)\left|\frac{\lambda_3}{2}\sqrt{\frac{16\pi}{5}}\right|^2\le \frac{7}{3}A^2,\\
\end{eqnarray*}
we can conclude that $3/7\le |A|^2$. This concludes the proof of \cref{lemma:rotation_case_Omega0}.
\end{proof}
\subsection{Proof of \cref{thm:main}}
The proof of \cref{thm:main} follows the same lines of the proof of \cref{thm:2}, with the main difference being that the degrees of freedom left after modding out the rotations are now three instead of one, see \cref{subsec:reduced_dof} below. This means that we need to use three Casimirs instead of one, and the non-degeneracy condition is more involved. 

To set up the steps of the proof, we consider a Rosby-Haurwitz wave of the form
$$
\omega_0^{RH} = \Omega x_3 + \sum_{m=-2}^2 \omega^{RH}_{2,m} Y_{2,m}.
$$
Our analysis distinguishes two cases, depending on the values of the Fourier coefficients $\omega^{RH}_{2,-1}$ and $\omega^{RH}_{2,1}$. Namely, we distinguish between the cases
\begin{equation}
    \sqrt{|\omega^{RH}_{2,-1}|^2+|\omega^{RH}_{2,1}|^2}\neq0,
    \label{eq:non_degenerate_case}
\end{equation}
and
\begin{equation}
    \sqrt{|\omega^{RH}_{2,-1}|^2+|\omega^{RH}_{2,1}|^2}=0.
    \label{eq:degenerate_case}
\end{equation}
The first case \eqref{eq:non_degenerate_case} is the non-degenerate one, where we can use three Casimirs to control the three degrees of freedom left after modding out the rotations, see \cref{non_degenerate}. This is akin to the situation in the proof of \cref{thm:2}. The second case \eqref{eq:degenerate_case} is the degenerate one, when the Jacobian of the map $F$ is rank deficient for any choice of Casimirs. Hence, to show stability we need to check conditions on higher order derivatives of the map $F$. In fact, we will show that choosing the Casimirs appropriately we are always in the case of a fold point and we can obtain a square root stability, see \cref{lemma:ker_1dim}. Within the degenerate case, we will have to further distinguish the salient zonal case that requires a specific treatment, see \cref{zonal_case}.

\subsubsection{Reduction of the degrees of freedom}\label{subsec:reduced_dof}
We start by showing how we reduce the degrees of freedom, by taking advantage of the invariance under rotations through the axis $p=e_3$. After applying \cref{lem:projection1}, we have reduced the problem to the study of $\tilde\omega_t$, which is given by a linear combination of the following spherical harmonics:
\begin{align}
    Y_{1,0}(\theta,\phi)=\sqrt{\frac{3}{4\pi}}\cos{\theta},\quad Y_{2,0}(\theta,\phi)=\sqrt{\frac{5}{16\pi}}(3\cos^2\theta-1),\nonumber \nonumber\\
    Y_{2,-1}(\theta,\phi)=\sqrt{\frac{15}{16\pi}}\sin(2\theta)\sin{\phi},\quad Y_{2,1}(\theta,\phi)=\sqrt{\frac{15}{16\pi}}\sin(2\theta)\cos{\phi}, \nonumber
   \\
    Y_{2,-2}(\theta,\phi)=\sqrt{\frac{15}{16\pi}}\sin^2\theta\sin(2\phi),\quad Y_{2,2}(\theta,\phi)=\sqrt{\frac{15}{16\pi}}\sin^2\theta\cos(2\phi).\label{eq:2ndshell}
\end{align}
We notice that both $Y_{1,0}$ and $Y_{2,0}$ are zonal, hence they are independent of rotation in the $\phi$ variable. On the other hand we have that
$$
Y_{2,-1}(\theta,\phi+\phi_0)=\cos\phi_0 Y_{2,-1}(\theta,\phi)+\sin\phi_0 Y_{2,-1}(\theta,\phi),\quad
Y_{2,1}(\theta,\phi+\phi_0)=-\sin\phi_0 Y_{2,-1}(\theta,\phi)+\cos\phi_0 Y_{2,-1}(\theta,\phi),
$$
and
$$
Y_{2,-2}(\theta,\phi+\phi_0)=\cos2\phi_0 Y_{2,-2}(\theta,\phi)+\sin2\phi_0 Y_{2,-2}(\theta,\phi),\quad
Y_{2,2}(\theta,\phi+\phi_0)=-\sin2\phi_0 Y_{2,-2}(\theta,\phi)+\cos2\phi_0 Y_{2,-2}(\theta,\phi).
$$
Hence, for any vorticity function of the form
$$
\omega(\theta,\phi)=\omega_{1,0}  Y_{1,0}(\theta,\phi)+\omega_{2,0} Y_{2,0}(\theta,\phi)+\omega_{2,1} Y_{2,1}(\theta,\phi)+\omega_{2,-1} Y_{2,-1}(\theta,\phi)+\omega_{2,-2} Y_{2,-2}(\theta,\phi)+\omega_{2,2} Y_{2,2}(\theta,\phi),
$$
there exist rotations $\hat\phi,\bar\phi\in [0,2\pi]$ such that
\begin{equation}\label{eq:1strep}
    \omega(\theta,\phi+\hat\phi)=\omega_{1,0}  Y_{1,0}(\theta,\phi)+\omega_{2,0} Y_{2,0}(\theta,\phi)+\hat\omega_{2,1} Y_{2,1}(\theta,\phi)+\hat\omega_{2,-2} Y_{2,-2}(\theta,\phi)+\hat\omega_{2,2} Y_{2,2}(\theta,\phi),
\end{equation}
and
\begin{equation}\label{eq:2ndrep}
\omega(\theta,\phi+\bar\phi)=\omega_{1,0}  Y_{1,0}(\theta,\phi)+\omega_{2,0} Y_{2,0}(\theta,\phi)+\bar\omega_{2,1} Y_{2,1}(\theta,\phi)+\bar\omega_{2,-1} Y_{2,-1}(\theta,\phi)+\bar\omega_{2,2} Y_{2,2}(\theta,\phi),
\end{equation}
for some values $\hat\omega_{2,1},\,\hat\omega_{2,-2},\,\hat\omega_{2,2},\,\bar\omega_{2,1},\,\bar\omega_{2,-1},\,\bar\omega_{2,2}$.
Furthermore, if we define $\alpha$ by the relationship
$$
\cos\alpha= \frac{\bar\omega_{2,1}}{\sqrt{\bar\omega^2_{2,1}+\bar\omega^2_{2,2}}},\qquad\sin\alpha= \frac{\bar\omega_{2,2}}{\sqrt{\bar\omega^2_{2,1}+\bar\omega^2_{2,2}}},
$$
we can express 
\begin{equation}\label{eq:3rdrep}
\omega(\theta,\phi+\bar\phi)=\omega_{1,0}  Y_{1,0}(\theta,\phi)+\omega_{2,0} Y_{2,0}(\theta,\phi)+\sqrt{\bar\omega^2_{2,1}+\bar\omega^2_{2,2}} Y_{2,1}(\theta,\phi+\alpha)+\bar\omega_{2,2} Y_{2,2}(\theta,\phi).
\end{equation}

\subsubsection{The non-degenerate case \eqref{eq:non_degenerate_case}}\label{non_degenerate}
We start in the same way to the case $\Omega=0$, by using \cref{lem:projection1} to reduce the problem to the study of the evolution of the Fourier coefficients in the second shell. For $\omega_0$ close enough to $\omega_0^{RH}$, we can write the projected solution as
\begin{equation*}
    \tilde\omega_t = \Omega x_3 + \sum_{m=-2}^2 \tilde\omega_{2,m}(t) Y_{2,m},  
\end{equation*}
with the normalization condition on the Fourier coefficients. By performing a rotation on the axis $p=e_3$, we can re-write the projected solution using the representation \eqref{eq:1strep}. Namely, there exists $u\in\R$ such that
\begin{equation}
    \tilde\omega_t(R^u_{e_3}x) = \Omega x_3 + A(t) Y_{2,0} + B(t) Y_{2,1} + C(t) Y_{2,-2} + D(t) Y_{2,2},
    \label{eq:representation_1_first_two_shells}
\end{equation}
With this representation, we consider the mapping $F:\mathbb{S}^4\to\erre^3$ defined by
\begin{equation}
    F(A,B,C,D)=\begin{pmatrix}C_3(\omega_{A,B,C,D}),C_4(\omega_{A,B,C,D}),C_5(\omega_{A,B,C,D})\end{pmatrix},
\end{equation}
where $C_k$ are the standard polynomial Casimirs and $\omega_{A,B,C,D}$ is the vorticity function defined by
\begin{equation*}
    \omega_{A,B,C,D}(x):=\Omega x_3 + A Y_{2,0}(x) + B Y_{2,1}(x) + C Y_{2,-2}(x) + D Y_{2,2}(x).
\end{equation*}
Explicitly, we can compute the Casimirs $C_3,C_4,C_5$ which are used to define the map $F$:
\begin{align*}
    \int_{\mathbb{S}^2}(\omega_{A,B,C,D})^3\;\mathrm{d}\mathcal{S}=&\frac{10 A^{3} + 15 \sqrt{3} B^{2} D + A \left(15 B^{2} - 30 C^{2} - 30 D^{2} + 56 \Omega^{2} \pi \right)}{14 \sqrt{5 \pi}}, \\
    \int_{\mathbb{S}^2}(\omega_{A,B,C,D})^4\;\mathrm{d}\mathcal{S}=&\frac{2}{7} \left( 11 A^{2} + 3 \left( 3 B^{2} + C^{2} + D^{2} \right) \right) \Omega^{2} + \frac{15 \left( A^{2} + B^{2} + C^{2} + D^{2} \right)^{2}}{28 \pi} + 4 \Omega^{4} \pi, \\
    \int_{\mathbb{S}^2}(\omega_{A,B,C,D})^5\;\mathrm{d}\mathcal{S}=&\sqrt{5}\frac{75 (A^{2} + B^{2} + C^{2} + D^{2}) \left( 2 A^{3} + 3 \sqrt{3} B^{2} D + 3 A \left( B^{2} - 2 (C^{2} + D^{2}) \right) \right)}{924 \pi^{3/2}}\\
    &+\sqrt{5}\frac{220 \left( 8 A^{3} + 3 B^{2} \left( 3 A + \sqrt{3} D \right) \right) \Omega^{2} \pi+ 1056 A \Omega^{4} \pi^{2}}{924 \pi^{3/2}}.
\end{align*}
To simplify the equation even further, we can use the normalization condition \eqref{eq:normalization_condition} to eliminate the $C$ variable $F(A,B,C,D)=F(A,B,\sqrt{1-A^2-B^2-D^2},D)=:G(A,B,D)$.

By \cref{lem:projection2} and the invariance of the Casimirs under rotations, we have that for all $t\geq0$ the following estimate holds
\begin{equation}
    |G(A(t),B(t),D(t))-G(A^{RH},B^{RH},D^{RH})|\leq C\|\omega_0-\omega^{RH}_0\|_2,
    \label{eq:final_Ck}
\end{equation}
where we are writing up to rotation
$$
\omega_0^{RH}(R^u_{e_3}x)=\Omega x_3 + A^{RH} Y_{2,0} + B^{RH} Y_{2,1} + C^{RH} Y_{2,-2} + D^{RH} Y_{2,2}.
$$
Computing the gradient of $G$, we have 
\begin{equation*}
    \det(\nabla G(A,B,D)) = \frac{480 \sqrt{3} \, B^3 \, \Omega^4 \left(-15 + 44 \Omega^2 \pi\right)}{3773 \pi},
\end{equation*}
which only vanishes when $B=0$ or $\Omega^2=\frac{15}{44\pi}$. Applying the inverse function theorem \cref{lemma:grad_neq0} we can conclude that there exists a constant $K>0$ and $\eps_0>0$ such that if $|(A(t),B(t),D(t))-(A^{RH},B^{RH},D^{RH})|<\eps_0$, then we have the estability estimate
\begin{equation*}
    |(A(t),B(t),D(t))-(A^{RH},B^{RH},D^{RH})|\leq C|G(A(t),B(t),D(t))-G(A^{RH},B^{RH},D^{RH})|.
\end{equation*}
Under the assumption that the initial perturbation is sufficiently small $0<\bar\eps<\frac{1}{K}\eps_0$, for some $K$ large enough, we can assert that
\begin{equation*}
|(A(0),B(0),D(0))-(A^{RH},B^{RH},D^{RH})|\le \|\omega_0-\omega_0^{RH}\|_2<\bar\eps\le \eps_0.
\end{equation*}
By continuity of the functions $A(t),B(t),D(t)$, we can conclude that $|(A(t),B(t),D(t))-(A^{RH},B^{RH},D^{RH})|<\bar\eps$ for $t$ small enough. Hence, we have that for $t$ small enough the following estimate holds
\begin{eqnarray*}
    \frac{1}{C}\inf_{s\in[0,\infty)}\|\tilde\omega_t-\omega^{RH}_s \|_2
    &\leq& |(A(t),B(t),D(t))-(A^{RH},B^{RH},D^{RH})|\\
    &\leq& C|G(A(t),B(t),D(t))-G(A^{RH},B^{RH},D^{RH})|\\
    &<& C c \|\omega_0 - \omega_0^{RH}\|_2\le C \bar\eps<\eps_0.
\end{eqnarray*}
Hence, we can conclude that $|(A(t),B(t),D(t))-(A^{RH},B^{RH},D^{RH})|<\eps_0$ for all $t\geq0$, and we can apply the inverse function theorem for all times. As long as
$$|\omega^{RH}_{2,1}|^2+|\omega^{RH}_{2,-1}|^2=|B^{RH}|^2\neq0\quad\text{and}\quad |\Omega|^2\neq \frac{15}{44\pi},$$
using \cref{lem:projection2} we have the desired stability estimate
\begin{equation*}
    \inf_{s\in[0,\infty)}\|\omega_t-\omega^{RH}_s\|_2\leq C\|\omega_0 - \omega_0^{RH}\|_2 .
\end{equation*}

We can get rid of the technical condition $|\Omega|^2\neq \frac{15}{44\pi}$ by considering different combinations of Casimirs to define the map $F$. Indeed, we can exhaust the domain of parameters $(\Omega,A,D)$ by considering different combinations of Casimirs.
Setting $|\Omega|^2=\frac{15}{44\pi}$, we compute
\begin{eqnarray*}
    \det(\nabla G_{\{C_3,C_4,C_7\}}(A,B,D))&=&\frac{3375 \sqrt{3}\, B^{3} }{1695694 \pi^{4}}\left( -\frac{18450}{121} + \frac{150}{11} \left( 9 + 8 A^{2} + 6 B^{2} \right) \right),\\
    \det(\nabla G_{\{C_3,C_4,C_6\}}(A,B,D))&=&-\frac{16875 \sqrt{5}\, B^{3}}{5934929 \pi^{7/2}}\left( 72 \sqrt{3}\, A^{3} + \sqrt{3}\, A \left( -\tfrac{984}{11} + 81 B^{2} \right) + 81 B^{2} D \right),\\
    \det(\nabla G_{\{C_3,C_4,C_8\}}(A,B,D))&=&\frac{3375 \sqrt{5}\, B^{3}}{14413399 \pi^{9/2}}\left( -\tfrac{22800}{11} \sqrt{3}\, A^{3} + \sqrt{3}\, A \left( \tfrac{32400}{11} - \tfrac{25650}{11} B^{2} \right) - \tfrac{25650}{11} B^{2} D \right),\\
    \det(\nabla G_{\{C_3,C_4,C_9\}}(A,B,D))&=&\frac{10125 B^{3}}{1095418324 \pi^{5}}\Big[-144000 \sqrt{3}\, A^{6} - 800 \sqrt{3}\, A^{4} \left( -\tfrac{3540}{11} + 405 B^{2} \right)\\
    &&- 30 \sqrt{3}\, A^{2} \left( \tfrac{36000}{121} - \tfrac{106200}{11} B^{2} + 6075 B^{4} \right) - 324000 A^{3} B^{2} D\\
    &&- 900 A B^{2} \left( -\tfrac{3540}{11} + 405 B^{2} \right) D+3 \sqrt{3} \left( -\tfrac{8424000}{1331} + \tfrac{2394000}{121} B^{2} - 20250 B^{4} D^{2} \right)\Big].
\end{eqnarray*}
For $B\neq 0$ there is no point in $\Omega,A,D$ where all the determinants vanish. Therefore, repeating the proof above with the appropriate combination of Casimirs, we can prove quantitative orbital stability for all $\Omega\neq0$ and $B\neq0$.

\subsubsection{The degenerate case \eqref{eq:degenerate_case}} \label{subsec:second_step}
We consider the degenerate case \eqref{eq:degenerate_case}, that is to say the case $\omega^{RH}_{2,-1}=\omega^{RH}_{2,1}=0$.

Again, we reduce the problem to studying the evolution of the Fourier coefficients in the second shell of $\tilde\omega_t$. By performing a rotation on the axis $p=e_3$, we re-write the projected solution using a different representation \eqref{eq:2ndrep} for convenience. Namely, there exists $u\in\R$ such that
\begin{equation}
    \tilde\omega_t(R^u_{e_3}x)=\Omega\cos\theta+ A(t)Y_{2,0} + B(t)Y_{2,1}+ C(t)Y_{2,-1} + D(t)Y_{2,2}, \label{eq:representation_2_first_two_shells}
\end{equation}
for some suitable functions of time $A,B,C,D$, with the normalization condition. Again with this representation, we consider the mapping $F:\mathbb{S}^4\to\erre^3$ defined by
\begin{equation*}
    F(A,B,C,D)=\begin{pmatrix}C_3(\omega_{A,B,C,D}),C_4(\omega_{A,B,C,D}),C_5(\omega_{A,B,C,D})\end{pmatrix},
\end{equation*}
where $C_k$ are the standard polynomial Casimirs and $\omega_{A,B,C,D}$ is the vorticity function defined by
\begin{equation*}
    \omega_{A,B,C,D}(x):=\Omega x_3 + A Y_{2,0}(x) + B Y_{2,1}(x) + C Y_{2,-1}(x) + D Y_{2,2}(x).
\end{equation*}
Explicitly, we can compute the Casimirs $C_3,C_4,C_5$ which are used to define the map $F$. Computing the integrals, we get
\begin{align*}
    \int_{\mathbb{S}^2}(\omega_{A,B,C,D})^3\;\mathrm{d}\mathcal{S}=&\frac{10 A^{3} + 15 \sqrt{3} (B - C) (B + C) D + 15 A \left( B^{2} + C^{2} - 2 D^{2} \right) + 56 A \Omega^{2} \pi}{14 \sqrt{5 \pi}}, \\
    \int_{\mathbb{S}^2}(\omega_{A,B,C,D})^4\;\mathrm{d}\mathcal{S}=&\frac{2}{7} \left( 11 A^{2} + 9 \left( B^{2} + C^{2} \right) + 3 D^{2} \right) \Omega^{2} + \frac{15 \left( A^{2} + B^{2} + C^{2} + D^{2} \right)^{2}}{28 \pi} + \frac{4 \Omega^{4} \pi}{5}, \\
    \int_{\mathbb{S}^2}(\omega_{A,B,C,D})^5\;\mathrm{d}\mathcal{S}=&\sqrt{5}\frac{75 (A^{2} + B^{2} + C^{2} + D^{2}) \left( 2 A^{3} + 3 \sqrt{3} (B - C)(B + C) D + 3 A \left( B^{2} + C^{2} - 2 D^{2} \right) \right)}{924 \pi^{3/2}}\\
    &+\sqrt{5}\frac{220 \left( 8 A^{3} + 9 A \left( B^{2} + C^{2} \right) + 3 \sqrt{3} (B - C)(B + C) D \right) \Omega^{2} \pi+ 1056 A \Omega^{4} \pi^{2}}{924 \pi^{3/2}}.
\end{align*}
Analogously to the previous section, we get rid of the $C$ variable by using the normalization condition \eqref{eq:normalization_condition}. Namely, we set $G(A,B,D):=F(A,B,\sqrt{1-A^2-B^2-D^2},D)$. In the case $\omega^{RH}_{2,-1}=\omega^{RH}_{2,1}=0$, under the representation \eqref{eq:2ndrep} we have that
$$
\omega^{RH}_0(R^u_{e_3}x)=\Omega x_3 + A^{RH} Y_{2,0}+ D^{RH} Y_{2,2},
$$
as $B^{RH}=C^{RH}=0$. By the normalization condition, this also implies that $|D^{RH}|^2=1-|A^{RH}|^2$.

The degeneracy of this case follows from the fact that the Jacobian $\nabla F(A^{RH},0,D^{RH})\in \R^{3\times 3}$ is always rank deficient; the middle column vanishes identically. Still, we can show that the rank is generically two away from the zonal case $|A^{RH}|^2=1$. Indeed, for $|A^{RH}|^2\ne 1$, $\Omega\ne 0$ and $|\Omega|^2\ne \frac{15}{176\pi}$, we have that $\mathrm{rank}(\nabla F(A^{RH},0,D^{RH}))=2$. This means that the kernel of the gradient is one-dimensional given by $\mathrm{span}\{e_2\}$. Computing, we can check that the fold point condition of \cref{lemma:ker_1dim} is satisfied, when $|A^{RH}|^2\ne 1$, $\Omega\ne 0$ and $|\Omega|^2\ne \frac{15}{176\pi}$. Applying \cref{lemma:ker_1dim}, there exists $\eps_0>0$ such that if $|(A(t),B(t),D(t))-(A^{RH},0,D^{RH})|<\eps_0$, then we can apply \cref{lemma:ker_1dim} to get
\begin{equation*}
    |(A(t),B(t),D(t))-(A^{RH},0,D^{RH})|\leq C|G(A(t),B(t),D(t))-G(A^{RH},0,D^{RH})|^{1/2}.
\end{equation*}
Under the assumption that the initial perturbation is sufficiently small $0<\bar\eps<\frac{1}{K}\eps_0^2$, for some $K$ large enough, we can ensure that the following estimates hold for all time. In fact, applying \cref{lem:projection2} and the invariance of the Casimirs under rotations, we have that for all $t\geq0$ the following estimate holds
\begin{eqnarray*}
    \frac{1}{C}\inf_{s\in[0,\infty)}\|\omega_t-\omega^{RH}_s\|_2&<&|(A(t),B(t),D(t))-(A^{RH},B^{RH},D^{RH})|\\
    &<&|G(A(t),B(t),D(t))-G(A^{RH},B^{RH},D^{RH})|^{1/2}\\
    &\leq& C\|\omega_0-\omega^{RH}_0\|_2^{1/2}.
    \label{eq:final_Ck_degenerate}
\end{eqnarray*}

To finish a full argument, we need to deal with are left with the cases $|A^{RH}|^2=1$, and $|\Omega|^2= \frac{15}{176\pi}$. The case $|A^{RH}|^2=1$ corresponds to the case of zonal flows, which we do in the next section as it requires a different representation of the solution after rotation. The last case $|\Omega|^2= \frac{15}{176\pi}$ can be treated by using different combinations of Casimirs to define the map $F$, as done in the previous section. In fact, computing we can show that for $|A^{RH}|^2\ne 1$ and $\Omega\ne 0$, the fold point condition of \cref{lemma:ker_1dim} is satisfied by either the map constructed above using the Casimirs $\{C_3,C_4,C_5\}$, or by the maps constructed using the Casimirs $\{C_3,C_4,C_6\}$ and $\{C_3,C_4,C_7\}$.

\subsubsection{The Zonal case $|\omega^{RH}_{2,0}|^2=1$} \label{zonal_case}
Finally, we are only left to address the case of a zonal flow
\begin{equation*}
    \bar\omega_0^{RH}=\Omega\cos\theta+\omega^{RH}_{2,0}Y_{2,0}.
\end{equation*}
As before, we reduce the problem to studying the evolution of the Fourier coefficients in the second shell of $\tilde\omega_t$. By performing a rotation on the axis $p=e_3$, we re-write the projected solution using a different representation \eqref{eq:3rdrep} for convenience. Namely, there exists $u\in\R$ such that
\begin{equation*}
    \tilde\omega_t(R_{e_3}^u x)=\Omega\cos\theta+ A(t) \sqrt{\frac{5}{16\pi}} \left( 3 \cos^2\theta - 1 \right) + B(t) \sqrt{\frac{15}{16\pi}} \sin(2\theta) \cos(\phi + \alpha(t)) + D(t)\sqrt{\frac{15}{16\pi}} \sin^2\theta \cos(2\phi),
\end{equation*}
for some suitable functions of time $A,B,D,\alpha$, with the normalization condition \eqref{eq:normalization_condition} $A^2+B^2+D^2=1$. In this representation, we consider the mapping $F:\mathbb{S}^2\times \R \to\erre^2$ defined only by the casimirs $C_3,C_4$ as before. Explicitly, we can compute the Casimirs $C_3,C_4$ which are used to define the map $F$. Computing the integrals, we get
\begin{align*}
    \int_{\mathbb{S}^2}(\omega_{A,B,D,\alpha})^3\;\mathrm{d}\mathcal{S}=&\frac{A \left(10 A^2 + 15 B^2 - 30 D^2 + 56 \Omega^2 \pi \right) + 15 \sqrt{3} B^2 D \cos(2\alpha)}{14 \sqrt{5 \pi}},\\
    \int_{\mathbb{S}^2}(\omega_{A,B,D,\alpha})^4\;\mathrm{d}\mathcal{S}=&\frac{2}{7} \left(11 A^2 + 9 B^2 + 3 D^2 \right) \Omega^2 + \frac{15 \left(A^2 + B^2 + D^2 \right)^2}{28 \pi} + \frac{4 \Omega^4 \pi}{5}.\\
\end{align*}
Noticing that the dependence on $B$ is only through $B^2$, we can use the normalization condition $B^2=1-A^2-D^2$, to define the smooth function $G(A,D,\alpha)=F(A,\sqrt{1-A^2-D^2},D,\alpha)$. By applying the invariance of the Casimirs under rotations and \cref{lem:projection2}, we have that for all $t\geq0$ the following estimate holds
\begin{eqnarray}
    |G(A(t),D(t),\alpha(t))-G(A^{RH},0,\alpha^{RH})|\leq C \|\omega_0-\omega_0^{RH}\|_2.
\end{eqnarray}
As we are in the case $B^{RH}=0$, we note that we have the freedom to choose $\alpha^{RH}\in \R$, as it does not appear in the expression of 
$$
\omega_0^{RH}=\Omega\cos\theta+A^{RH}\sqrt{\frac{5}{16\pi}}(3\cos^2\theta-1).
$$ 
For each $t\geq0$, it is convenient to choose $\alpha^{RH}=\alpha(t)$ to simplify the analysis of stability. So we are interested in studying the Jacobian 
$$
\nabla_{A,D} F|_{(A,D)=(\pm 1,0)}=\begin{pmatrix}
\frac{8\Omega^2}{7} A^{RH} &0\\
0 & 0
\end{pmatrix},
$$ 
whose second column vanishes identically and has exactly rank 1 if $\Omega\ne 0$. We can check that the hypothesis and the fold point condition of \cref{lemma:ker_1dim} are satisfied uniformly in the value $\alpha$ and $\Omega\ne 0$. Therefore, we obtain that there exists $C>0$ and $\eps_0>0$ (uniform in $\alpha$) such that if $|(A(t),D(t))-(A^{RH},0)|<\eps_0$, then we can apply \cref{lemma:ker_1dim} to get the estimate
$$
    |(A(t),D(t))-(A^{RH},0)|\leq C|G(A(t),D(t),\alpha(t))-G(A^{RH},0,\alpha(t))|^{1/2}.
$$
Picking the initial perturbation sufficiently small $0<\bar\eps<\frac{1}{K}\eps_0^2$, for some $K$ large enough, we can ensure that the following estimates hold for all time. In fact, applying \cref{lem:projection2} and the invariance of the Casimirs under rotations, we have that for all $t\geq0$ the following estimate holds
\begin{eqnarray*}
    \frac{1}{C}\|\omega_t-\omega^{RH}_0\|_2&<&|(A(t),D(t))-(A^{RH},0)|\\
    &<&|G(A(t),D(t),\alpha(t))-G(A^{RH},0,\alpha(t))|^{1/2}\\
    &\leq& C\|\omega_0-\omega^{RH}_0\|_2^{1/2},
    \label{eq:final_Ck_zonal}
\end{eqnarray*}
provided that $A^{RH}=\pm1$, and $\Omega\ne 0$.
Here, we have used that the value of $\alpha(t)$ is irrelevant to bound the distance $\|\omega_t-\omega^{RH}_0\|_2$, as we are showing that $|B(t)|\le C\|\omega_0-\omega^{RH}_0\|_2^{1/2}$. This concludes the proof of Theorem \ref{thm:main}.

\subsection{Remarks and open problems}
We close by pointing out some observations about the argument employed in the proof of \cref{thm:main}. An intriguing matter regards the sharpness of the rate $\frac{1}{2}$, which is also an open problem in the case of the torus studied by Elgindi \cite[Section 1.2]{elgindi2023remarkstabilityenergymaximizers}.

In principle, the singularity points of the Jacobian of the mapping may depend both on the chosen representation of the vorticity function (see \cref{subsec:reduced_dof}), and on the chosen combination of Casimirs. We try to shed some light on this matter by fixing the representation
\begin{align}\label{eq:rmkrep}
    \tilde{\omega}_t=\Omega\cos\theta+ \sin\beta \sqrt{\frac{5}{16\pi}} \left( 3 \cos^2\theta - 1 \right)+ \cos\beta\sin\gamma \sqrt{\frac{15}{16\pi}} \sin(2\theta) \cos(\phi + \alpha) + \cos\beta\cos\gamma\sqrt{\frac{15}{16\pi}} \sin^2\theta \cos(2\phi),
\end{align}
for some suitable functions of time $\alpha,\beta,\gamma:[0,\infty)\to\esse^1$. The observation is that any mapping formed by Casimirs has a rank deficient Jacobian, if $|\omega_{2,-1}|=|\omega_{2,1}|=0$ corresponding to $\sin\gamma=0$. Indeed, given any combination of Casimirs $f=(f_1,f_2,f_3)\in\mathcal{C}^1(\erre;\erre^3)$, consider the mapping $F:\esse^1\times\esse^1\times\esse^1\to\erre^3$ given by
\begin{equation*}
    F(\alpha,\beta,\gamma)=\left(\int_{\esse^2}f_1\left(\omega_{\alpha,\beta,\gamma}\right)\;\mathrm{d}\mathcal{S},\int_{\esse^2}f_2\left(\omega_{\alpha,\beta,\gamma}\right)\;\mathrm{d}\mathcal{S},\int_{\esse^2}f_3\left(\omega_{\alpha,\beta,\gamma}\right)\;\mathrm{d}\mathcal{S}\right).
\end{equation*}
We notice that the column of the Jacobian $\nabla F(\alpha,\beta,\gamma)$ with respect to the variable $\alpha$ always vanishes when $\sin\gamma=0$, since
\begin{equation*}
    0=\partial_\alpha F_k(\alpha,\beta,\gamma)=-\int_{\esse^2}f_k'\biggl(\mathbb{P}_*\mathbb{P}_{\leq2}\omega(\alpha,\beta,\gamma)\biggr)\cos\beta\sin\gamma \sqrt{\frac{15}{16\pi}} \sin(2\theta) \sin(\phi + \alpha)\;\mathrm{d}\mathcal{S}.
\end{equation*}
Summing up, in the representation \eqref{eq:rmkrep} the case $|\omega_{2,-1}|=|\omega_{2,1}|=0$ is always a singular point for the Jacobian of the constructed mapping, no matter the choice of Casimirs.

\section*{Acknowledgements}
The authors thank Michele Coti Zelati for useful discussions and relentless guidance. LM acknowledges support from the Engineering and Physical Sciences Research Council (EPSRC). The research of MGD was partially supported by NSF-DMS-2205937. The computations in this paper were carried out using Wolfram Mathematica.

\bibliographystyle{abbrvnat}
\bibliography{ref}
\end{document}